\documentclass[a4paper]{amsart}

\usepackage[T1]{fontenc}
\usepackage{amsmath, amssymb, amsthm, mathrsfs, mathtools, graphicx}
\usepackage{varioref}
\usepackage[backref=page]{hyperref}
\usepackage{enumitem, xspace, ifthen, comment}
\usepackage[all]{xy}
\usepackage[nameinlink, capitalize]{cleveref}
\usepackage[dvipsnames]{xcolor}
\usepackage{tikz}

\setlist[enumerate]{
  label=(\thethm.\arabic*),
  before={\setcounter{enumi}{\value{equation}}},
  after={\setcounter{equation}{\value{enumi}}},
  itemsep=1ex
}

\setlist[itemize]{
  leftmargin=*,
  topsep=1ex,
  itemsep=1ex,
  label=$\circ$
}

\newcommand{\mypagesize}{
  \addtolength{\textwidth}{20pt}
  \addtolength{\textheight}{14pt}
  \calclayout
}
\mypagesize

\newtheorem*{thm-plain}{Theorem}
\newtheorem{thm}{Theorem}[section]
\newtheorem{lem}[thm]{Lemma}
\newtheorem{prp}[thm]{Proposition}
\newtheorem{cor}[thm]{Corollary}

\newtheorem{ques}[thm]{Question}

\numberwithin{equation}{thm}

\theoremstyle{definition}
\newtheorem{dfn}[thm]{Definition}

\newtheorem*{dfn-plain}{Definition}

\theoremstyle{remark}
\newtheorem{clm}[thm]{Claim}

\newtheorem{awlog}[thm]{Additional Assumption}

\newtheorem{setup}[thm]{Setup and Notation}
\newtheorem{rem}[thm]{Remark}

\newtheorem{exm}[thm]{Example}
\newtheorem*{rem-plain}{Remark}

\newcommand{\inv}{^{-1}}
\newcommand{\from}{\colon}

\newcommand{\lto}{\longrightarrow}

\newcommand{\x}{\times}
\newcommand{\inj}{\hookrightarrow}

\newcommand{\bij}{\xrightarrow{\,\smash{\raisebox{-.5ex}{\ensuremath{\scriptstyle\sim}}}\,}}
\newcommand{\isom}{\cong}
\newcommand{\defn}{\coloneqq}
\newcommand{\ndef}{\eqqcolon}
\newcommand{\tensor}{\otimes}

\newcommand{\wt}{\widetilde}

\renewcommand{\d}{\mathrm d}

\newcommand{\dual}{^{\smash{\scalebox{.7}[1.4]{\rotatebox{90}{\textup\guilsinglleft}}}}}
\newcommand{\ddual}{^{\smash{\scalebox{.7}[1.4]{\rotatebox{90}{\textup\guilsinglleft} \hspace{-.5em} \rotatebox{90}{\textup\guilsinglleft}}}}}

\newcommand{\factor}[2]{\left. \raise 2pt\hbox{$#1$} \right/\hskip -2pt \raise -2pt\hbox{$#2$}}

\newdir{ ir}{{}*!/-5pt/@^{(}} % Injective arrow, for \xymatrix{}.
\newdir{ il}{{}*!/-5pt/@_{(}} % Usage: \ar@{ ir->}, \ar@{ il->} }}

\DeclareMathOperator{\rk}{rk}

\newcommand{\lref}{\labelcref}

\renewcommand*{\qed}[1][\textsf{Easy.}]{%
\leavevmode\unskip\penalty9999 \hbox{}\nobreak\hfill
    \quad\hbox{#1}%
}

\newcommand{\set}[1]{\left\{ #1 \right\}}

\def\rd#1.{\lfloor{#1}\rfloor}
\def\rp#1.{\lceil{#1}\rceil}
\def\tw#1.{\langle{#1}\rangle}

\renewcommand{\O}[1]{\mathscr{O}_{#1}}
\newcommand{\Omegap}[2]{\Omega_{#1}^{#2}}
\newcommand{\Omegat}[2]{\check\Omega_{#1}^{#2}}
\newcommand{\Omegar}[2]{\Omega_{#1}^{[#2]}}
\newcommand{\T}[1]{\mathscr{T}_{#1}}
\newcommand{\can}[1]{\omega_{#1}}

\newcommand{\reg}[1]{{#1}_{\mathrm{reg}}}
\newcommand{\sing}[1]{{#1}_{\mathrm{sg}}}

\newcommand{\codim}[2]{\mathrm{codim}_{#1}(#2)}

\newcommand{\cc}[2]{\mathrm{c}_{#1}(#2)}

\def\Hnought#1.#2.{\mathit{\Gamma} \!\left( #1, #2 \right)}
\def\HH#1.#2.#3.{\mathrm{H}^{#1} \!\left( #2, #3 \right)}
\def\hh#1.#2.#3.{h^{#1} \!\left( #2, #3 \right)}
\def\RR#1.#2.#3.{R^{#1} #2_* #3}
\def\HHc#1.#2.#3.{\mathrm{H}_{\mathrm{c}}^{#1} \!\left( #2, #3 \right)}
\def\Hh#1.#2.#3.{\mathrm{H}_{#1} \!\left( #2, #3 \right)}
\def\Hom#1.#2.{\mathrm{Hom} \!\left( #1, #2 \right)}
\def\End#1.{\mathrm{End} \!\left( #1 \right)}
\def\sHom#1.#2.{\mathscr{H}\!om \!\left( #1, #2 \right)}
\def\Ext#1.#2.#3.{\mathrm{Ext}^{#1} \!\left( #2, #3 \right)}
\def\sExt#1.#2.#3.{\mathscr{E}\!xt^{#1} \!\left( #2, #3 \right)}
\def\Link#1.#2.{\mathrm{Link} \!\left( #1, #2 \right)}

\newcommand{\piet}[1]{\hat\pi_1(#1)}

\newcommand{\GL}[2]{\mathrm{GL}(#1, #2)}

\newcommand{\kahler}{K{\"{a}}hler\xspace}
\newcommand{\cy}{Calabi--Yau\xspace}
\newcommand{\lt}{locally trivial\xspace}
\newcommand{\gkp}{maximally quasi-\'etale\xspace}
\newcommand{\qe}{quasi-\'etale\xspace}
\newcommand{\lz}{Lipman--Zariski Conjecture\xspace}

\DeclareMathOperator{\tor}{tor}

\DeclareMathOperator{\Alb}{Alb}

\DeclareMathOperator{\Deflt}{Def^{lt}}

\renewcommand{\theta}{\vartheta}
\renewcommand{\phi}{\varphi}

\newcommand{\N}{\ensuremath{\mathbb N}}

\newcommand{\R}{\ensuremath{\mathbb R}}
\newcommand{\C}{\ensuremath{\mathbb C}}

\renewcommand{\frm}{\mathfrak m}

\newcommand{\frX}{\mathfrak X} 

 \newcommand{\sE}{\mathscr E} \newcommand{\sF}{\mathscr F}

\definecolor{forrest}{RGB}{81,133,49}
\definecolor{mydarkblue}{RGB}{10,92,153}

\title[A decomposition theorem for singular K{\"{A}}hler spaces]{A decomposition theorem for singular K{\"{A}}hler \\ spaces with trivial first Chern class \\ of dimension at most four}

\dedicatory{\raggedright
\hspace{.084\linewidth} Nemo: \\
\hspace{.084\linewidth} Ad litteram nemo: \\
\hspace{.084\linewidth} Ne anima quidem ulla: \\
\hspace{.084\linewidth} Graf: Cum non inutile sit eandem veritatem per methodos diversas perscrutari \dots}

\author{Patrick Graf}
\address{Fachbereich 08, Johannes-Gutenberg-Universit\"at Mainz, Staudingerweg 9, 55099 Mainz, Germany}
\email{\href{mailto:patrick.graf@uni-bayreuth.de}{patrick.graf@uni-bayreuth.de}}
\urladdr{\href{http://www.pgraf.uni-bayreuth.de/en/}{www.graficland.de}}

\date{January 17, 2021}

\hypersetup{
  pdfauthor={Patrick Graf},
  pdftitle={A decomposition theorem for singular Kahler spaces with trivial first Chern class of dimension at most four},
  pdfkeywords={Kahler spaces, klt singularities, vanishing first Chern class, unobstructed deformations, decomposition theorem},
  pdfstartview={Fit},
  pdfpagelayout={OneColumn},
  pdfpagemode={UseNone},
  linktoc=all,
  breaklinks,
  linkcolor=[RGB]{0 0 96},
  citecolor=[RGB]{96 0 0},
  urlcolor=[RGB]{0 96 0},
  colorlinks}

\begin{document}

\begin{abstract}
Let $X$ be a compact \kahler fourfold with klt singularities and vanishing first Chern class, smooth in codimension two.
We show that $X$ admits a Beauville--Bogomolov decomposition: a finite quasi-\'etale cover of $X$ splits as a product of a complex torus and singular Calabi--Yau and irreducible holomorphic symplectic varieties.
We also prove that $X$ has small projective deformations and the fundamental group of $X$ is projective.
To obtain these results, we propose and study a new version of the Lipman--Zariski conjecture.
\end{abstract}

\maketitle

% \begingroup
% \hypersetup{linkcolor=black}
% \tableofcontents
% \endgroup

\section{Introduction}

Let $X$ be a compact \kahler manifold such that $\cc1X = 0 \in \HH2.X.\R.$.
The structure of such $X$ is in many aspects well-understood~\cite{Beauville83, Bogomolov78, Tian87, Todorov89}:
\begin{itemize}
\item Geometry: a finite \'etale cover of $X$ splits as a product of a complex torus, simply connected \cy manifolds and irreducible holomorphic symplectic manifolds (Beauville--Bogomolov decomposition).
\item Deformation theory: The local deformation space of $X$ is smooth (Bogomolov--Tian--Todorov theorem), and $X$ admits small projective deformations (Kodaira problem).
\item Topology: the fundamental group $\pi_1(X)$ is projective, virtually abelian, and finite if the augmented irregularity $\wt q(X)$ of $X$ vanishes.
\end{itemize}

The BB (= Beauville--Bogomolov) decomposition is a cornerstone in the classification of compact \kahler manifolds up to biholomorphic maps.
The subject of birational geometry, however, is rather a structure theory up to \emph{bimeromorphic} maps.
As is well-known, in this context it is necessary to consider singular analogues of the above manifolds.
By this, we mean compact \kahler spaces with klt singularities and $\cc1X = 0$.
In the projective case, the BB decomposition has been established by~\cite{Druel18, GGK19, HoeringPeternell19}.
This result is commonly referred to as the BBDGGHKP decomposition.
But most of the other properties listed in the beginning remain elusive even for projective varieties.

Very recently, the decomposition theorem has been extended to the \kahler case~\cite{Kodairaflat, BochnerCGGN, BakkerGuenanciaLehn20}.
The final statement may therefore be called the
\[ \text{BBBCDGGGHKLNPS decomposition}. \]
The goal of this paper is to give an independent proof of the decomposition in dimension four.
Our principal result, however, is a generalization of the BTT (= Bogomolov--Tian--Todorov) theorem.

\begin{thm}[Singular BTT theorem in dimension four] \label{btt}
Let $X$ be a normal compact \kahler space of dimension $\le 4$, with klt singularities and such that $\cc1X = 0 \in \HH2.X.\R.$.
Assume that $\dim \sing X \le 1$.
Then the semiuniversal \lt deformation space $\Deflt(X)$ is smooth, unless possibly in the case $\wt q(X) = \dim X - 3$.
\end{thm}

In all cases where \cref{btt} does not apply, $X$ is already projective.
This allows us to draw the following consequences.

\begin{cor}[Kodaira problem in dimension four] \label{kod}
Let $X$ be as in \cref{btt}.
Then the semiuniversal family $\frX \to \Deflt(X)$ is a strong \lt algebraic approximation.
In particular, $X$ is locally algebraic.
\end{cor}

This puts~\cite[Thm.~H]{BochnerCGGN} in its final form:
we remove both the local algebraicity assumption and the necessity to take a \qe cover before obtaining an algebraic approximation of $X$.
Note also that \cref{kod} confirms~\cite[Conj.~K]{BochnerCGGN} in this particular case.

We now turn to fundamental groups.
In~\cite[Thm.~G]{BochnerCGGN}, we showed that in dimension four $\pi_1(X)$ is virtually abelian.
In particular, this group is ``virtually projective'', i.e.~it contains a normal subgroup of finite index which is isomorphic to $\pi_1(Y)$ for some projective manifold $Y$.
\cref{pi1} below shows that passing to a subgroup is in fact not necessary.
It also implies that $\pi_1$ of any \kahler fourfold of Kodaira dimension zero admitting a good minimal model is projective.
This gives some new evidence towards the conjecture that every \kahler group is projective~\cite[(1.26)]{ABCKT96}.

Campana's Abelianity Conjecture~\cite[Conj.~7.3]{Cam04} also makes similar predictions for the fundamental group of the smooth locus $\pi_1(\reg X)$, but this is currently out of reach even for $X$ projective.
Using \cref{kod}, we can at least confirm the conjecture for the profinite completion $\piet{\reg X}$.

\begin{cor}[Fundamental groups] \label{pi1}
Let $X$ be as in \cref{btt}.
Then:
\begin{enumerate}
\item\label{pi1.1} The fundamental group $\pi_1(X)$ is projective.
\item\label{pi1.2} The algebraic fundamental group of the smooth locus $\piet{\reg X}$ is virtually abel\-ian, and finite if $\wt q(X) = 0$.
\end{enumerate}
\end{cor}

Finally, we have the decomposition theorem mentioned in the title.
For the definition of singular CY and IHS varieties, we refer to~\cite[Def.~6.11]{BochnerCGGN}.

\begin{cor}[BB decomposition] \label{bb}
Let $X$ be as in \cref{btt}.
Then some \qe cover of $X$ splits as a product of a complex torus, \cy and irreducible holomorphic symplectic varieties.
\end{cor}

\subsection*{The flat factor of $\T X$}

Let us briefly comment on the proof of \cref{btt}.
It relies on an analysis of the flat factor $\sF$ in the holonomy decomposition of $\T X$.
To explain this, and also to fix notation, recall that in~\cite[Thm.~C]{BochnerCGGN} we proved in particular the following.

\begin{setup}[Standard setting] \label{std}
Let $X$ be a normal compact \kahler space with klt singularities such that $\cc1X = 0 \in \HH2.X.\R.$.
Then after replacing $X$ by a finite \qe cover, the so-called \emph{holonomy cover}, the tangent sheaf of $X$ decomposes as
\[ \T X = \sF \oplus \bigoplus_{k \in K} \sE_k, \]
where the sheaves $\sF$ and $\sE_k$ satisfy the following:
\begin{itemize}
\item The restriction $\sF\big|_{\reg X}$ is flat, i.e.~given by a representation $\pi_1(\reg X) \to \mathrm{SU}(r)$, where $r = \rk(\sF)$.
\item Each summand $\sE_k\big|_{\reg X}$ has full holonomy group either $\mathrm{SU}(n_k)$ or $\mathrm{Sp}(n_k/2)$, with respect to a suitable singular Ricci-flat metric.
Here $n_k = \rk(\sE_k) \ge 2$.
\end{itemize}
\end{setup}

The natural conjecture concerning the flat factor $\sF$ is that it corresponds to the torus factor in the (conjectural) Beauville--Bogomolov decomposition of $X$.
It is convenient to rephrase this in different but equivalent ways:
\begin{itemize}
\item The rank of $\sF$ should equal the augmented irregularity of $X$, i.e.~$r = \wt q(X)$.
\item If $\wt q(X)$ vanishes, then $\sF$ should be the zero sheaf.
\end{itemize}
The conjecture follows from~\cite{BakkerGuenanciaLehn20}, but for the purpose of giving a logically independent proof of \cref{btt}, the following partial result is key.
It bounds the rank of $\sF$ from above in terms of the dimension of the singular locus of $X$ and in fact holds in arbitrary dimension.

\begin{thm}[Bounding the flat factor] \label{main bound}
In the standard setting~\lref{std}, assume that the augmented irregularity of $X$ vanishes, $\wt q(X) = 0$.
Then:
\begin{enumerate}
\item\label{mb.1} The flat factor $\sF$ satisfies $\rk(\sF) \le \dim \sing X$.
\item\label{mb.2} If $\dim \sing X \le 1$, then $\sF = 0$.
\end{enumerate}
\end{thm}

\begin{rem-plain}
In~\lref{mb.1}, we need to adopt the convention that the empty set has dimension zero (as opposed to $-1$ or $-\infty$) in order for the conclusion to hold also if $X$ is smooth.
\end{rem-plain}

\subsection*{The \lz for direct summands}

To prove \cref{main bound}, we propose and study a new variant of the \lz, which we explain now.
The classical \lz (which is still open) states that if the tangent sheaf $\T X$ of a complex algebraic variety or complex space is locally free, then $X$ is smooth.
Here we ask what happens if $\T X$ is not necessarily locally free, but contains a locally free direct summand.

\begin{ques}[\lz for direct summands] \label{directlz ques}
Let $X$ be a complex space.
Assume that the tangent sheaf of $X$ admits a direct sum decomposition
\[ \T X = \sE \oplus \sF, \]
where $\sF$ is locally free.
Under what assumptions on the rank of $\sF$ and on the singularities of $X$ can we conclude that $X$ is smooth?
\end{ques}

\cref{directlz ques} is formulated in a deliberately vague way.
In this work, we will concentrate on the case of klt singularities, but other classes of singularities would be equally interesting.

\begin{thm}[\cref{directlz ques} for klt singularities] \label{directlz}
Let $X$ be a normal complex space with klt singularities such that the tangent sheaf admits a direct sum decomposition
\[ \T X = \sE \oplus \sF, \]
where $\sF$ is locally free of rank $r$.
If $\dim \sing X \le r - 1$, then $X$ is smooth.
\end{thm}

Easy examples show that the bound on $r$ in \cref{directlz} is sharp (\cref{sharp}).

\section{Notation and basic facts}

Unless otherwise stated, complex spaces are assumed to be countable at infinity, separated, reduced and connected.
Algebraic varieties and schemes are always assumed to be defined over the complex numbers.

\begin{dfn}[Torsion-free differentials]
Let $X$ be a reduced complex space and $p \in \N$ a non-negative integer.
The sheaf of \emph{torsion-free differential $p$-forms} on $X$ is defined to be
\[ \Omegat Xp \defn \factor{\Omegap Xp}{\tor \Omegap Xp}, \]
where $\Omegap Xp \defn \bigwedge^p \Omegap X1$ is the sheaf of \kahler differentials and $\tor \Omegap Xp$ is the subsheaf of $\Omegap Xp$ consisting of the sections vanishing on some dense open subset $U \subset X$.
Equivalently, $\tor \Omegap Xp$ consists of those sections whose support is contained in the singular locus $\sing X$.
\end{dfn}

\begin{dfn}[Quasi-\'etale covers]
A \emph{cover} is a finite, surjective morphism $\gamma \from Y \to X$ of normal, connected complex spaces.
A cover $\gamma$ is called \emph{\qe} if there exists a closed subset $Z \subset Y$ with $\codim{Y}{Z} \ge 2$ such that $\gamma\big|_{Y \setminus Z} \from Y \setminus Z \to X$ is \'etale.
\end{dfn}

\begin{dfn} \label{gkp}
Let $X$ be a normal complex space.
A \emph{\gkp cover} of $X$ is a \qe Galois cover $\gamma \from Y \to X$ satisfying the following equivalent conditions:
\begin{enumerate}
\item\label{gkp.1} Any \'etale cover of $\reg Y$ extends to an \'etale cover of $Y$.
\item\label{gkp.2} Any \qe cover of $Y$ is \'etale.
\item\label{gkp.3} The natural map of \'etale fundamental groups $\piet{\reg Y} \to \piet Y$ induced by the inclusion $\reg Y \inj Y$ is an isomorphism.
\end{enumerate}
\end{dfn}

\begin{dfn}[Irregularity]
The \emph{irregularity} of a compact complex space $X$ is $q(X) \defn \hh1.Y.\O Y.$, where $Y \to X$ is any resolution of singularities.
The \emph{augmented irregularity} of $X$ is
\[ \wt q(X) \defn \max \Big\{ q \big( \wt X \big) \;\Big|\; \wt X \to X \text{ quasi-\'etale} \Big\} \in \N_0 \cup \{ \infty \}. \]
\end{dfn}

\subsection*{Vector fields}

Let $X$ be a reduced complex space.
A \emph{vector field} on $X$ is a (local) section of the \emph{tangent sheaf} $\T X \defn \sHom \Omegap X1.\O X.$, where $\Omegap X1$ is the sheaf of \kahler differentials.

Let $v$ be a vector field.
For any point $x \in X$, the germ of $v$ at $x$ is a \C-linear derivation $v_x \from \O{X,x} \to \O{X,x}$ and hence it can be restricted to an element of the Zariski tangent space of $X$ at $x$:
\[ v(x) \in \left( \factor{\frm_x}{\frm_x^2} \right) \dual \ndef T_x X. \]
Note, however, that in general not every Zariski tangent vector at $x$ is of the form $v(x)$ for some local vector field $v$.

\begin{lem} \label{506}
Let $Z \subset X$ be an analytic subset that is fixed by every local automorphism of $X$.
Then $v(z) \in T_z Z \subset T_z X$ for every vector field $v$ on $X$ defined near $z \in Z$.
\end{lem}

\begin{proof}
We will use the correspondence between derivations, vector fields and local \C-actions as described in~\cite[\S 1.4, \S 1.5]{Akh95}.
The vector field $v$ induces a local \C-action $\Phi \from \C \x X \to X$.
By the definition of local group action, $\Phi(t, -)$ is an automorphism of germs $(X, z) \bij \big( X, \Phi(t, z) \big)$ for every sufficiently small $t \in \C$.
It then follows from the assumption that $\Phi(t, z) \in Z$ for every $t \in \C$.
Now, we can recover the derivation $\delta$ corresponding to $v$ from $\Phi$ by the formula
\begin{equation} \label{kaup}
\delta(f)(x) = \frac\d{\d t}\bigg|_{t=0} f \big( \Phi(t, x) \big)
\end{equation}
for every $f \in \O{X,z}$.
Plugging the above statement into~\lref{kaup}, we see that $\delta$ stabilizes the ideal of $Z$, i.e.~$\delta(I_{Z,z}) \subset I_{Z,z}$.
Hence $\delta$ induces a derivation of $\O{Z,z}$ and then also an element of $T_z Z$.
\end{proof}

\begin{exm}
Let $X = \C^2$ and let $0 \in C \subset X$ be a curve such that $0$ is a singular point of $C$.
If $v$ is a vector field on $X$ that is tangent to $C \setminus \set0$, then its local flows restrict to automorphisms of $C$.
These automorphisms necessarily fix the singular point $0 \in \sing C$ and hence $v$ vanishes at the origin, $v(0) = 0$.
\end{exm}

\subsection*{Deformation theory}

This is just a very quick reminder.
For more details, see for example~\cite{Kodairaflat}.

\begin{dfn}[Deformations of complex spaces]
A \emph{deformation} of a (reduced) compact complex space $X$ is a proper flat morphism $\pi \from \frX \to (S, 0)$ from a (not necessarily reduced) complex space $\frX$ to a complex space germ $(S, 0)$, equipped with a fixed isomorphism $\frX_0 \defn \pi\inv(0) \isom X$.
\end{dfn}

\begin{dfn}[Algebraic approximations] \label{def alg approx}
Let $X$ be a compact complex space and $\pi \from \frX \to S$ a deformation of $X$.
Consider the set of projective fibres
\[ S^{\mathrm{alg}} \defn \big\{ s \in S \;\big|\; \frX_s \text{ is projective} \big\} \subset S \]
and its closure $\overline{S^{\mathrm{alg}}} \subset S$.
We say that $\frX \to S$ is an \emph{algebraic approximation of $X$} if $0 \in \overline{S^{\mathrm{alg}}}$.
We say that $\frX \to S$ is a \emph{strong algebraic approximation of $X$} if $\overline{S^{\mathrm{alg}}} = S$ as germs, i.e.~$S^{\mathrm{alg}}$ is dense near $0 \in S$.
\end{dfn}

\begin{dfn}[Locally trivial deformations] \label{dfn lt def}
A deformation $\pi \from \frX \to S$ is called \emph{\lt} if for every $x \in \frX_0$ there exist open subsets $0 \in S^\circ \subset S$ and $x \in U \subset \pi\inv(S^\circ)$ and an isomorphism
\[ \xymatrix{
U \ar^-\sim[rr] \ar_-\pi[dr] & & (\frX_0 \cap U) \x S^\circ \ar^-{\operatorname{pr}_2}[dl] \\
& S^\circ. &
} \]
\end{dfn}

\section{The \lz for direct summands}

In this section, we prove \cref{directlz} from the introduction and the following corollary.
We then deduce \cref{main bound}.

\begin{cor}[Spaces with large flat summands] \label{directlz-cor}
Let $X$ be a normal complex space with klt singularities such that
\[ \T X = \sE \oplus \sF, \]
where $\sF\big|_{\reg X}$ is flat of rank $r$, i.e.~given by a representation $\pi_1(\reg X) \to \GL r\C$.
If $\dim \sing X \le r - 1$, then $X$ has only quotient singularities.
\end{cor}

\begin{rem}[Sharpness of \cref{directlz}] \label{sharp}
The bound on $\dim \sing X$ in \cref{directlz} is sharp, as shown by the (easy) example $X = Y \x \C$, where $Y$ is a (non-smooth) isolated klt singularity.
In this case $\T X$ has a rank one free summand and $\dim \sing X = 1$, but $X$ is not smooth.
\end{rem}

\begin{rem}[Reformulation of \cref{directlz}]
The conclusion of \cref{directlz} could also be formulated in a somewhat oblique manner as follows: every irreducible component of $\sing X$ has dimension at least $r$.
\end{rem}

\begin{rem}[Comparison to previous results]
The ``usual'' \lz is well-known for spaces with klt, or even log canonical, singularities~\cite[Cor.~1.3]{GK13}.
If $r = n \defn \dim X$ in \cref{directlz}, the statement reduces to this result.

In general, $\dim \sing X = n - 2$ and then the only case left where \cref{directlz} applies is $r = n - 1$.
In this case, we may consider a (local) index one cover $X_1 \to X$.
There, also the rank one sheaf $\sE$ will become locally free, hence $X_1$ is smooth.
Summing up, we see that~\cite{GK13} only yields the weaker statement that $X$ has quotient singularities (instead of being smooth).
\end{rem}

\subsection{Proof of \cref{directlz}}

Assuming that $Z \defn \sing X$ is non-empty, we will derive a contradiction.
Let $f \from Y \to X$ be the functorial resolution.
Pick a sufficiently general point $z \in Z$.
Then $Z$ is smooth at $z$, i.e.~$z \in \reg Z$, and by Generic Smoothness the fibre $F \defn f\inv(z) \subset Y$ will be a simple normal crossings variety (albeit not necessarily a divisor in $Y$).
Shrinking $X$ around $z$, we may without loss of generality make the following
\begin{awlog}
The singular locus $Z$ of $X$ is smooth.
The sheaf $\sF$ is free, isomorphic to $\O X^{\oplus r}$.
\end{awlog}
Let $\set{v_1, \dots, v_r}$ be a basis of $\sF$ and let $\set{\alpha_1, \dots, \alpha_r}$ be the dual basis of $\sF \dual$, defined by $\alpha_i(v_j) = \delta_{ij}$.
Since $\sF \subset \T X$ is a direct summand, so is $\sF \dual \subset \Omegar X1$.
This enables us to consider the sections $\alpha_i$ as reflexive $1$-forms on $X$.
Evaluating the vector fields $v_i$ at the point $z$ and taking into account that $Z$ is stabilized by their flows, we obtain $v_i(z) \in T_z Z$.
Since $\dim T_z Z = \dim Z \le r - 1$, there is a non-trivial relation
\begin{equation} \label{linabh}
\phantom{\lambda_i \in \C \qquad} \sum_{i=1}^r \lambda_i v_i(z) = 0 \in T_z Z \subset T_z X, \qquad \lambda_i \in \C.
\end{equation}
Some coefficient in~\lref{linabh}, say $\lambda_1$, will be non-zero.
Replacing $v_1$ by $\sum_{i=1}^r \lambda_i v_i$, we arrive at the
\begin{awlog}
The free sheaf $\sF$ has a basis $\set{v_1, \dots, v_r}$ with the property that $v_1(z) = 0$.
\end{awlog}
This means that $z$ is stabilized by the flow of $v_1$.
Let $\wt v_1$ be the lift of $v_1$ to $Y$.
Then the flow of $\wt v_1$ stabilizes the fibre $F = f\inv(z)$, i.e.~$\wt v_1$ restricts to a vector field on $F$.
The same is then true of any irreducible component of $F$.
Fix one such component $P \subset F$, and note that $P$ is smooth because $F$ is an snc variety.

On the other hand, let $\wt \alpha_1$ be the lift of $\alpha_1$ to $Y$, which is a holomorphic $1$-form by~\cite{KebekusSchnell18}.
Since the function $\wt \alpha_1(\wt v_1)$ is identically one, the restricted form $\wt \alpha_1 \big|_P$ cannot be zero.
Indeed, if $p \in P$ is arbitrary then $\wt \alpha_1$ evaluated on $\wt v_1(p) \in T_p P$ is non-zero.
On the other hand, the restriction map factors as
\begin{equation} \label{restr factor}
\HH0.Y.\Omegap Y1. \lto \underbrace{\HH0.F.\Omegat F1.}_{= 0} \lto \HH0.P.\Omegap P1.
\end{equation}
because $P$ is smooth and not contained in the singular locus of $F$.
The middle term vanishes due to~\cite[Cor.~1.5]{HaconMcKernan07} and \cite[Thm.~4.1]{Kebekus13}.
Hence~\lref{restr factor} implies $\wt \alpha_1\big|_P = 0$, contradicting our previous observation that this form is non-zero and thus ending the proof. \qed

\subsection{Proof of \cref{directlz-cor}}

Since the problem is local, we may assume $X$ to be a germ.
We argue by induction on $\dim \sing X$.
If $\sing X = \emptyset$, there is nothing to show.
Otherwise, by~\cite[Prop.~5.8]{BochnerCGGN} there exists an open subset $\reg X \subset X^\circ \subset X$ admitting a \gkp cover $\gamma^\circ \from Y^\circ \to X^\circ$ and satisfying $\dim(X \setminus X^\circ) \le {\dim \sing X - 1}$.
We may extend $\gamma^\circ$ to a \qe cover $\gamma \from Y \to X$, by~\cite[Thm.~3.4]{DethloffGrauert94}.
\[ \xymatrix{
Y^\circ \ar@{ ir->}[rr] \ar_-{\gamma^\circ}[d] & & Y \ar^-\gamma[d] \\
X^\circ \ar@{ ir->}[rr] & & X
} \]
Note that $Y$ reproduces all assumptions of \cref{directlz-cor}.
In particular, taking the reflexive pullback of the decomposition $\T X = \sE \oplus \sF$, we get
\[ \T Y = \sE_Y \oplus \sF_Y \]
where $\sF_Y\big|_{\reg Y}$ is flat.
Restricting $\sF_Y$ to $\reg Y^\circ$ and using property~\lref{gkp.3} combined with~\cite{Platonov68}, it follows that $\sF_Y\big|_{Y^\circ}$ is flat, in particular locally free.
Thus $Y^\circ$ is smooth thanks to \cref{directlz}.
This implies
\[ \dim \sing Y \le \dim(Y \setminus Y^\circ) = \dim(X \setminus X^\circ) \le \dim \sing X - 1. \]
By the induction hypothesis, $Y$ has only quotient singularities and hence so does~$X$. \qed

\subsection{Proof of \cref{main bound}}

Assume that $\rk(\sF) \ge \dim \sing X + 1$.
By \cref{directlz-cor}, the space $X$ has only quotient singularities.
In particular, it is locally algebraic and hence $\rk(\sF) = \wt q(X) = 0$ by~\cite[Thm.~C]{BochnerCGGN}.
This contradiction proves~\lref{mb.1}.

For~\lref{mb.2}, we only need to exclude the case that $\rk(\sF) = 1$.
In this case, $\sF\big|_{\reg X}$ is given by a representation $\pi_1(\reg X) \to \mathrm{SU}(1) = \set 1$, so it is the trivial line bundle.
By reflexivity, this implies $\sF \isom \O X$.
In particular, $\HH0.X.\sF \dual. \ne 0$.
Since $\sF \dual \subset \Omegar X1$, this contradicts the assumption that $\wt q(X) = 0$. \qed

\section{Proof of main results}

\subsection{Torus covers revisited}

In the standard setting~\lref{std}, recall the notion of \emph{torus cover} from~\cite[Thm.~B]{BochnerCGGN}:
this is a \qe cover $\gamma \from T \x Z \to X$, where $T$ is a complex torus of dimension $\wt q(X)$, while $Z$ satisfies $\can Z \isom \O Z$ as well as $\wt q(Z) = 0$.
We do not know if the map $\gamma$ can always be chosen to be \qe.
Indeed, this is not obvious from the construction and taking the Galois closure of a given $\gamma$ might destroy the splitting property.
The following weaker statement is however sufficient for our purposes:

\begin{prp}[Torus covers revisited] \label{tcr}
In the standard setting~\lref{std}, the torus cover can be chosen in such a way that it is a composition of \qe Galois morphisms:
\[ T \x Z \xrightarrow{\quad\textup{Galois}\quad} X' \xrightarrow{\quad\textup{Galois}\quad} X. \]
\end{prp}

\begin{proof}
Consider the index one cover $X_1 \to X$ and choose a \qe cover $X_2 \to X_1$ with $q(X_2) = \wt q(X)$.
Replacing $X_2 \to X$ by its Galois closure $X' \to X$ yields the first map in the statement to be proven.

For the second map, we know from~\cite[Thm.~4.1]{BochnerCGGN} that the Albanese map $X' \to A \defn \Alb(X')$ becomes trivial after a finite \'etale base change $A_1 \to A$.
We may consider the Galois closure $A_2 \to A$ of the latter map.
The pullback of $X'$ to $A_2$ still splits, by the transitivity of fibre products.
\[ \xymatrix{
F \x A_2 \ar[rr] \ar[d] & & F \x A_1 \ar[rr] \ar[d] & & X' \ar[d] \\
A_2 \ar[rr] & & A_1 \ar[rr] & & A.
} \]
Furthermore, $F \x A_2 \to X'$ is Galois, being the pullback of the Galois morphism $A_2 \to A$ along $X' \to A$.
We now set $T \defn A_2$ and $Z \defn F$.
The proof that this is indeed a torus cover of $X$ is the same as in~\cite[proof of Cor.~4.2]{BochnerCGGN}.
\end{proof}

\begin{lem} \label{def tc}
Notation as above.
\begin{enumerate}
\item\label{def tc.1} The tangent space $T_0 \Deflt(T \x Z)$ can be calculated as follows:
\[ \HH1.T \x Z.\T{T \x Z}. = \HH1.T.\T T. \oplus \HH1.Z.\T Z.. \]
\item\label{def tc.2} If $\Deflt(Z)$ is smooth, then so is $\Deflt(T \x Z)$.
More precisely, in this case $\Deflt(T \x Z) = \Deflt(T) \x \Deflt(Z)$.
\end{enumerate}
\end{lem}

\begin{proof}
Let $p \from T \x Z \to T$ and $r \from T \x Z \to Z$ be the projections.
The proof is in a series of claims.

\begin{clm} \label{T tc}
The tangent sheaf of $T \x Z$ decomposes as
\[ \T{T \x Z} = p^* \T T \oplus r^* \T Z. \]
\end{clm}

\begin{proof}
Clearly the decomposition exists on the smooth locus of $T \x Z$.
By reflexivity, it extends to a decomposition $\T{T \x Z} = p^{[*]} \T T \oplus r^{[*]} \T Z$, where $p^{[*]} \T T \defn (p^* \T T)\ddual$ and $r^{[*]} \T Z$ denote the reflexive pullback.
Hence it suffices to show that $p^* \T T$ and $r^* \T Z$ are already reflexive.
For $p^* \T T$, this is obvious because $\T T$ is even (locally) free.

For $r^* \T Z$, we use the characterization of reflexive sheaves as locally 2\textsuperscript{nd} syzygy sheaves, \cite[Ch.~I, Lemma~1.1.16 and proof of Lemma~1.1.10]{OSS80}\footnote{The cited reference assumes the underlying space to be smooth, but the arguments work verbatim for sheaves on normal complex spaces.}.
That is, on sufficiently small open sets $U \subset Z$ there exists an exact sequence
\[ 0 \lto \T Z\big|_U \lto \O U^{\oplus n} \lto \O U^{\oplus m}. \]
The sequence stays exact when pulled back along the flat morphism $r$, showing that also $r^* \T Z\big|_{T \x U}$ is reflexive.
Of course, this argument shows quite generally that the pullback of a reflexive sheaf via a flat map remains reflexive.
\end{proof}

\begin{clm} \label{proj formula}
For any $i \ge 0$, we have
\[ \RR i.p.p^* \T T. = \T T \tensor \RR i.p.\O{T \x Z}. \]
and
\[ \RR i.r.r^* \T Z. = \T Z \tensor \RR i.r.\O{T \x Z}.. \]
\end{clm}

\begin{proof}
For $\T T$, this is simply the projection formula.
Regarding $\T Z$, some care is required because that sheaf is not locally free.
So let $U \subset Z$ be a Stein open subset.
Then on the level of presheaves, we need to check that
\begin{equation} \label{827}
\HH i.r\inv(U).r^* \T Z. = \HH0.U.\T Z. \tensor \HH i.r\inv(U).\O{T \x Z}.,
\end{equation}
where the tensor product is taken over the ring $\HH0.U.\O Z.$.
But $r\inv(U) = T \x U$, so the K\"unneth formula tells us that the left-hand side equals
\begin{equation} \label{832}
\bigoplus_{k + \ell = i} \HH k.U.\T Z. \tensor \HH\ell.T.\O T.
= \HH0.U.\T Z. \tensor \HH i.T.\O T.
\end{equation}
since the higher cohomology groups on $U$ vanish.
For the same reason, the second factor on the right-hand side is
\begin{equation} \label{836}
\bigoplus_{k + \ell = i} \HH k.U.\O Z. \tensor \HH\ell.T.\O T.
= \HH0.U.\O Z. \tensor \HH i.T.\O T.
= \HH i.T.\O T..
\end{equation}
\lref{827} now follows by comparing~\lref{832} and~\lref{836}.
\end{proof}

By \cref{T tc}, statement~\lref{def tc.1} is reduced to the following claim.

\begin{clm} \label{coh tc}
$\HH1.T \x Z.p^* \T T. = \HH1.T.\T T.$ and $\HH1.T \x Z.r^* \T Z. = \HH1.Z.\T Z.$.
\end{clm}

\begin{proof}
Concerning the first factor, the Leray spectral sequence combined with \cref{proj formula} gives
\[ 0 \lto \HH1.T.\T T. \lto \HH1.T \x Z.p^* \T T. \lto \HH0.T.{\T T \tensor \RR1.p.\O{T \x Z}.}., \]
where the last term vanishes because $q(Z) = 0$.
We obtain that $\HH1.T \x Z.p^* \T T. = \HH1.T.\T T.$.

For the second factor, again by \cref{proj formula} we have a similar exact sequence
\[ 0 \lto \HH1.Z.\T Z. \lto \HH1.T \x Z.r^* \T Z. \lto \HH0.Z.{\T Z \tensor \RR1.r.\O{T \x Z}.}. \]
but here $\RR1.r.\O{T \x Z}.$ is a trivial vector bundle (of rank equal to $\dim T$).
To conclude as before, we therefore need to know that $\HH0.Z.\T Z. = 0$.
This can be seen as follows, where $m = \dim Z$ and $\wt Z \to Z$ is a resolution:
\begin{align*}
\hh0.Z.\T Z. & = \hh0.Z.\Omegar Z{m - 1}. && \text{contraction and $\can Z \isom \O Z$} \\
& = \hh0.\wt Z.\Omegap{\wt Z}{m - 1}. && \text{\cite[Cor.~1.8]{KebekusSchnell18}} \\
& = \hh m - 1.\wt Z. \O{\wt Z}. && \text{Hodge theory on $\wt Z$} \\
& = \hh m - 1.Z.\O Z. && \text{$Z$ has rational singularities} \\
& = \hh1.Z.\O Z. && \text{Serre duality~\cite[Ch.~VII, Thm.~3.10]{BS76}} \\
& = 0 && \text{because $q(Z) = \wt q(Z) = 0$.}
\end{align*}
This ends the proof of \cref{coh tc}.
\end{proof}

For~\lref{def tc.2}, assume that $\Deflt(Z)$ is smooth.
Recall that $\Deflt(T)$ is smooth in any case because $T$ is a complex torus.
Hence also $B \defn \Deflt(T) \x \Deflt(Z)$ is smooth.
Consider the product deformation of $T \x Z$ over $B$, i.e.~the fibre over a point $(t, s) \in B$ is $T_t \x Z_s$.
The Kodaira--Spencer map of this deformation,
\[ \kappa \from T_0 B \lto \HH1.T \x Z.\T{T \x Z}., \]
is an isomorphism by~\lref{def tc.1}.
In particular, it is surjective and this implies that $\Deflt(T \x Z)$ is smooth.
Furthermore, our deformation over $B$ is pulled back from the semiuniversal deformation via a map
\[ B \lto \Deflt(T \x Z), \]
which on tangent spaces induces the isomorphism $\kappa$.
Hence the map itself is likewise an isomorphism.
This ends the proof.
\end{proof}

\subsection{Proof of \cref{btt}}

Keeping notation, note that $\dim Z = \dim X - \wt q(X)$, so that we may assume $\dim Z \ne 3$ for the purpose of this proof.
Note also that smoothness of $\Deflt(Z)$ implies smoothness of $\Deflt(X)$ by \cref{def tc}, \cref{tcr} and~\cite[Prop.~5.3]{Kodairaflat}.
We may therefore replace $X$ by $Z$ and then we only need to show the following: in the standard setting~\lref{std}, if $\wt q(X) = 0$ and $\dim X \in \set{1, 2, 4}$ then $\Deflt(X)$ is smooth.

If $\dim X \le 2$, then $X$ has only quotient singularities and hence the claim follows directly from~\cite[Cor.~1.7]{Kodairaflat}.
If $\dim X = 4$, then we apply \cref{main bound} to conclude that the flat factor of $\T X$ vanishes, $\sF = 0$.
(Note that here we have implicitly replaced $X$ by a cover, but this is harmless by~\cite[Prop.~5.3]{Kodairaflat} again because that cover can be taken to be Galois.)
Consequently, there are only two possibilities for the holonomy decomposition of $\T X$.
In slightly abused notation, they are:
\begin{itemize}
\item $\mathrm{Sp}(2)$, the strongly stable case, and
\item $\mathrm{SU}(2) \oplus \mathrm{SU}(2)$, the non-stable case.
\end{itemize}
Since $\mathrm{SU}(2) = \mathrm{Sp}(1)$, in both cases $X$ carries a holomorphic symplectic form.
This implies smoothness of $\Deflt(X)$ by~\cite[Thm.~4.7]{BakkerLehn18}, cf.~also the footnote in the proof of~\cite[Thm.~8.4]{BochnerCGGN}. \qed

\subsection{Proof of \cref{kod}}

Again let $T \x Z \to X$ be a torus cover.
If $\dim Z \ne 3$, then $\Deflt(X)$ is smooth by \cref{btt}.
In this case, \cref{kod} follows directly from~\cite[Thm.~1.2]{Kodairaflat}.
We may therefore assume that $\dim Z = 3$.
In this case, $\hh2.Z.\O Z. = \hh1.Z.\O Z. = q(Z) = 0$ by Serre duality.
Since also $\hh2.T.\O T. = 0$ for dimension reasons, we conclude from the K\"unneth formula that $\HH2.T \x Z.\O{T \x Z}. = 0$.
This in turn implies $\HH2.X.\O X. = 0$.
This Hodge number is constant in \lt families, as one can see e.g.~by performing a simultaneous resolution and using that $X$ has rational singularities.
That is, every \lt deformation $X_t$ of $X$ still satisfies $\HH2.X_t.\O{X_t}. = 0$.
By the Kodaira embedding theorem, $X_t$ is projective for any $t$ (including $X_t = X$).
In particular, any \lt deformation of $X$ is a strong algebraic approximation. \qed

\subsection{Proof of Corollaries~\lref{pi1} and~\lref{bb}}

The projectivity of $\pi_1(X)$ follows from Thom's First Isotopy Lemma, which in our situation says that $X$ and $X_t$ are homeo\-morphic.
For more details, cf.~the proof of~\cite[Cor.~1.8]{AlgApprox}.
Thus~\lref{pi1.1} is proved.

Regarding~\lref{pi1.2}, we know from \cref{kod} that $X$ is locally algebraic and hence it admits a \gkp cover $\gamma \from Y \to X$ by~\cite[Prop.~5.9]{BochnerCGGN}.
By~\lref{gkp.3}, the map $\piet{\reg Y} \to \piet Y$ is an isomorphism.
Also, $\piet{\reg Y} = \piet{\gamma\inv(\reg X)} \to \piet{\reg X}$ is injective with image of finite index.
It is therefore sufficient to prove the claim for $\piet{\reg Y}$.
But this reduces to~\cite[Thm.~G]{BochnerCGGN} applied to $Y$.

We already remarked above that we now know that $X$ is locally algebraic.
Therefore \cref{bb} follows immediately from~\cite[Thm.~H]{BochnerCGGN}. \qed

\newcommand{\etalchar}[1]{$^{#1}$}
\providecommand{\bysame}{\leavevmode\hbox to3em{\hrulefill}\thinspace}
\providecommand{\MR}{\relax\ifhmode\unskip\space\fi MR}
\providecommand{\MRhref}[2]{%
  \href{http://www.ams.org/mathscinet-getitem?mr=#1}{#2}
}
\providecommand{\href}[2]{#2}

\end{document}